\documentclass[11pt]{article}

\usepackage{pgf,tikz}
\usepackage{mathrsfs}
\usetikzlibrary{arrows}

\usepackage{graphicx} 
\graphicspath{{figures/}} 

\usepackage{fullpage, url,amsmath,amsfonts,amssymb,mathtools,mathrsfs,graphicx,  algorithm, float, sansmath,epstopdf,color,caption,enumitem,tabularx}
\usepackage[final]{pdfpages}
\usepackage{amsthm}
\usetikzlibrary{automata,topaths}
\usetikzlibrary{decorations.pathreplacing,shapes.misc}
\usepackage{fancyhdr}
\usepackage{pgf,tikz}
\usepackage{mathrsfs}
\usetikzlibrary{arrows}

\usepackage{todonotes}
\usepackage{amsmath} 
\usepackage{verbatim}
\usepackage{amssymb, float,empheq}
\usepackage{color}
\usepackage{graphicx,epsfig}
\graphicspath{{figures/}}
\usepackage[applemac]{inputenc}
\usepackage{graphicx}

\usepackage{graphicx} 
\graphicspath{{figures/}} 

\def\beq{\begin{equation}}
\def\eeq{\end{equation}}
\def\baq{\begin{eqnarray}}
\def\eaq{\end{eqnarray}}
\def\baqn{\begin{eqnarray*}}
\def\eaqn{\end{eqnarray*}}

\newcommand{\ball}{\mathbb{B}}

\usepackage{multirow}
\usetikzlibrary{calc,arrows}
\theoremstyle{plain}

\newtheorem{remark}{Remark}

\newtheorem{example}{Example}
\newtheorem{theorem}{Theorem}
\newtheorem{lemma}[theorem]{Lemma}

\usepackage{blindtext}

\usepackage[colorlinks,linkcolor=blue,citecolor=red]{hyperref}

\usepackage{xcolor, framed}

\newcommand{\R}{{\mathbb R}}
\newcommand{\N}{{\mathbb N}}

\newcommand{\interior}{{\rm int}\kern 0.06em}

\def\<{\langle}
\def\>{\rangle}

\usepackage{ulem}

\usepackage{subcaption}

\usepackage{pict2e}

\if
{

\usepackage[pageref]{backref}
\renewcommand*{\backrefalt}[4]{%
\ifcase #1 %
(Not cited)%
\or
(Cited on p.~#2)%
\else
(Cited on pp.~#2)%
\fi
}

}
\fi

 \usepackage{booktabs}

\begin{document}
\title{Reducing the Lipschitz Constant in Tseng Algorithm and the Convergence Analysis}
\author{
 Ba Khiet Le \thanks{Analytical and Algebraic Methods in Optimization Research Group, Faculty of Mathematics and Statistics, Ton Duc Thang University, Ho Chi Minh City, Vietnam.\vskip 0mm
 E-mail: \texttt{lebakhiet@tdtu.edu.vn}}\qquad
 Minh Thuan Tran \thanks{Faculty of Mathematics and Statistics, Ton Duc Thang University, Ho Chi Minh City, Vietnam. E-mail: c2300078@student.tdtu.edu.vn}\qquad
 Thi Thanh Tien Pham \thanks{Faculty of Mathematics and Statistics, Ton Duc Thang University, Ho Chi Minh City, Vietnam. E-mail: c2300083@student.tdtu.edu.vn}\qquad
  Kim Thuong Vo  \thanks{Faculty of Mathematics and Statistics, Ton Duc Thang University, Ho Chi Minh City, Vietnam. E-mail: c2300081@student.tdtu.edu.vn}
 }

\maketitle

\begin{abstract}
{In this paper, we prove} the weak and linear convergence of Tseng algorithm applied to the zero of  sum of two operators problem where each operator is not necessarily monotone while  existing researches in the literature require the monotonicity of both operators. Consequently it allows us to decompose more efficiently and can reduce the Lipschitz constant significantly, which is meaningful especially in big data. Numerical examples supporting the theoretical results are also provided. 
\end{abstract}


\section{Introduction}
Many optimization problems (see, e.g., \cite{Bauschke,br,cp,cw,Mordukhovich,Mordukhovich24,Nesterov3,Rockafellar,roc-wets}) can be reduced into the monotone inclusion 
\beq\label{main0}
 0\in \mathcal{A}x,
 \eeq
 where $A: H \rightrightarrows H$ is a set-valued maximally monotone operator and $H$ is a Hilbert space. It is the case of minimizing a proper lower semi-continuous convex function $f: H\to \R\cup\{\infty\}$ 
 \beq
 \min_{x\in H} f(x)
 \eeq
 where the minimizers satisfy (\ref{main0}) with $\mathcal{A}=\partial f$, the convex subdifferential of $f$. We are interested in the splitting case where $\mathcal{A}$ can be split as the sum of two maximally monotone operators $A : H \rightrightarrows H$ and $B:  H \rightrightarrows H$:
\beq\label{main}
 0\in Ax+Bx,
 \eeq
  For example the constrained optimization problem  
  \beq
 \min_{x\in C} f(x)= \min_{x\in H} f(x)+I_C(x),
 \eeq
  where $C$ is a convex set and $I_C$ denotes the indicator function of $C$, can be rewritten into this form. Indeed, under some qualification assumption,  the necessary optimality condition becomes
  \beq
 0\in \partial f (x)+N_C(x),
 \eeq
 where $N_C(\cdot)$ denotes the normal cone operator to the convex set $C$, which is maximally monotone. To solve (\ref{main}) numerically, especially if both $A$ and $B$ are set-valued, we can use the Douglas-Rachford (DR) splitting algorithm (see, e.g., \cite{Bauschke,Bauschke1,cp,Davis,Giselsson,Giselsson1,He,Hong,Moursi}) to obtain the weak and linear convergence, which has the following form
\beq
{x_0\in H},\;\;x_{k+1}= [(1-\alpha)Id+\alpha R_{\gamma A} R_{\gamma B}](x_k),\; {k\in \N}
\eeq
for some $\alpha \in (0, 1]$ and $\gamma>0$. This result comes from the fact that (\ref{main})  can be rewritten as follows (see, e.g., \cite{Bauschke})
\beq
\begin{cases}
x=J_{\gamma {B}}(z),\\
z= R_{\gamma A}R_{\gamma B}(z),
\end{cases}
\eeq
where $J_{\gamma {B}}:=(Id+\gamma {B})^{-1}$ and $R_{\gamma B}:=2J_{\gamma {B}}-Id$ denote the resolvent and the reflected resolvent of $\gamma {B}$ repsectively. 
If $B$ is single-valued and Lipschitz continuous (for example if $B=\nabla f$ where $f$ is a smooth function), we can use the Forward-Backward (FB)  algorithm
\beq
{x_0\in H},\;\;x_{k+1}=  J_{\gamma A} (x_k-\gamma Bx_k),\; {k\in \N},\;\gamma>0.
\eeq
 The FB algorithm  converges linearly if $A$ or $B$ is strongly monotone and converges strongly if $B$ is cocoercive (see, e.g., \cite{Chen,cw,Lions}, see also  \cite{Apidopoulos,Attouch} for some inertial versions). However it may diverge if the strong monotonicity or cocoercivity is lacked. For example we can consider in $\R^2$ with 
 $$A=0 \;{\rm and}\; Bx=\left( \begin{array}{ccc}
x^{(2)} \\ \\
-x^{(1)}
\end{array} \right)\;\; {\rm with}\;\;x=\left( \begin{array}{ccc}
x^{(1)} \\ \\
x^{(2)}
\end{array} \right).
$$
 Then 
 $$
 \Vert x_{k+1}\Vert^2=(1+\gamma^2) \Vert x_{k}\Vert^2
 $$
 and thus the sequence $(x_k)$ is unbounded, diverges if $x_0\neq 0$. In order to improve the convergence property, Tseng \cite{Tseng} modified the Forward-Backward algorithm as follows 
 \beq\label{det}
\begin{cases}
y_{k}=J_{\gamma {A}}(x_k-\gamma Bx_k),\\
x_{k+1}= y_k+\gamma (Bx_k-By_k),
\end{cases}
\eeq
 to receive the weak convergence of $(x_k)$ to a solution of (\ref{main}).  While existing researches using Tseng algorithm require the  monotonicity  of both $A$ and $B$ (see, e.g., \cite{Bauschke,Gibali,Thong,Wang,Yang} and the references therein),  in this paper we prove the weak convergence of Tseng algorithm under only the  monotonicity of $A+B$. It means that $A$ or $B$ can be non-monotone and we can have more flexibility to get  better efficiency. In particular, it can reduce the Lipschitz constant significantly (see Section \ref{sec5}), which is a desirable property in big data where the Lipschitz constant is usually very big. In the case $A+B$ is maximally strongly monotone, the linear convergence rate is obtained. 

The paper is organized as follows. First we recall some standard notation and results in the monotone operators theory in Section \ref{sec2}. The weak and linear convergence of  Tseng algorithm under the maximal monotonicity of the sum  are proved in Section \ref{sec3}. Some numerical examples are provided in \ref{sec5}. Finally we end the paper in \ref{sec6}.

\section{Notations and preliminaries} \label{sec2}
Let $H$ be a given {real} Hilbert space with {the inner product $\langle \cdot,\cdot \rangle$ and the associated norm $\Vert \cdot \Vert$.} The closed unit ball of $H$ is denoted by $\ball$. Let $C$ be a closed convex subset of $H$ and $x\in C$. 
The distance and the projection from a point $s$ to $C$ are defined respectively by 
$${ d}(s,C):=\inf_{x\in C} \|s-x\|, \;\;{\rm Proj}_C(s):={\bar{x}} \in C \;\;{\rm such \;that \;} { d}(s,C)= \|s-{\bar{x}}\|.$$
The normal cone  to  $C$ at $x\in C$ is defined  by
$$
N_C(x):=\{x^*\in H: \langle x^*, y-x \rangle\le 0,\;\;\forall\;y\in C\}.
$$

\noindent  A  mapping  $A: H\to H$ is called $\mu$-strongly monotone ($\mu>0$) provided
$$
\langle  Ax-Ay, x-y \rangle \ge \mu\Vert x-y \Vert^2\;\;\forall\;x, y\in H.
$$
It is called $L$-Lipschitz continuous ($L>0$) if
$$
\Vert Ax-Ay\Vert \le L \Vert x-y \Vert\;\;\forall\;x, y\in H.
$$

\noindent The domain, the range and the graph of a set-valued mapping $\mathcal{F}: {H}\rightrightarrows {H}$ are defined respectively by
$${\rm dom}(\mathcal{F})=\{x\in {H}:\;\mathcal{F}(x)\neq \emptyset\},\;\;{\rm rge}(\mathcal{F})=\displaystyle\bigcup_{x\in{H}}\mathcal{F}(x)\;\;$$
 and
 $$\;\;{\rm gph}(\mathcal{F})=\{(x,y): x\in{H}, y\in \mathcal{F}(x)\}.$$

\noindent {It is called monotone provided

$$
\langle x^*-y^*,x-y \rangle \ge 0 \;\;\forall \;x, y\in H, x^*\in \mathcal{F}(x) \;{\rm and}\;y^*\in \mathcal{F}(y).
$$
In addition, if there is no monotone operator $\mathcal{G}$ such that the graph of $\mathcal{F}$ is strictly {included} in  the graph of $\mathcal{G}$, then $\mathcal{F}$ is called maximally monotone.  Next we recall the Opial's Lemma, a useful result to prove the weak convergence of sequences. 
\begin{lemma}[Opial's Lemma]
\label{l:Opial}
Let \( S \) be a nonempty subset of \( \mathcal{H} \), and let \( (x_k)_{k \in \mathbb{N}} \) be a sequence in \( \mathcal{H} \). Suppose that:  
\begin{enumerate}
\item
For every \( x^* \in S \), \( \lim_{k \to \infty} \|x_k - x^*\| \) exists.  
\item
Every sequential weak cluster point of \( (x_k)_{k \in \mathbb{N}} \) belongs to \( S \).  
\end{enumerate}
Then the sequence \( (x_k)_{k \in \mathbb{N}} \) converges weakly to some point \( x_\infty \in S \).  
\end{lemma}
Finally we have the following result which allows us to compute the resolvent of a sum of a normal cone operator and a positive  semi-definite matrix. The extension to the sum of a maximally monotone operator and a positive  semi-definite matrix can be done similarly. 
\begin{lemma}\label{sumn}
Let $y=J_{N_K+A}(x)$ where $K$ is a closed convex set in $\R^n$ and $A\in \R^{n\times n}$ is a positive semi-definite matrix. Suppose that  $Id+A=B^TB$, where $B\in \R^{n\times n}$ is invertible. Let $P=(B^T)^{-1}$. Then 
$$
 \Leftrightarrow y=P^{T}{\rm Proj}_D(Px),
$$
where the set $D=BK=\{Bk: k\in K\}$.
\end{lemma}
\begin{proof}
We have 
\baqn
&&y=J_{N_K+A}(x)\Leftrightarrow x=N_K(y)+(Id+A)y=N_K(y)+B^TBy\\
&& \Leftrightarrow (B^T)^{-1}x=(B^T)^{-1}N_K(B^{-1}By)+By\\
&& \Leftrightarrow Px=PN_K(P^Tz)+z\;\;{\rm where\;} z=By\\
&& \Leftrightarrow  Px= N_D(z)+z \;\;{\rm where\;}D= (P^T)^{-1} K=BK\\
&& \Leftrightarrow z=(N_D+Id)^{-1}(Px)={\rm Proj}_D(Px)\\
&& \Leftrightarrow y=B^{-1}{\rm Proj}_D(Px)=P^{T}{\rm Proj}_D(Px).
\eaqn
\begin{remark}
Note that if $A=0$ then we obtain the classical result that $J_{N_K}={\rm Proj}_K.$
\end{remark}
\end{proof}
\section{Main result}\label{sec3}
In this section, we prove the weak and linear convergence of Tseng algorithm under only the  monotonicity of the sum, not requiring the  monotonicity of each operator. 
\subsection{Weak Convergence} 

First let us propose the following assumptions. \\

\noindent$\mathbf{Assumption \;1}:$ The single-valued mapping $B: H\to H$ is $L$-Lipschitz continuous for some $L>0$. \\

\noindent$\mathbf{Assumption \;2}:$ The set-valued mapping $A: {H}\rightrightarrows {H}$ is such that $J_{\gamma {A}}$ is well-defined single-valued on $H$  and $A+B$ is maximally monotone.\\

Tseng algorithm has the form as follows.

\beq\label{det}
\begin{cases}
y_{k}= J_{\gamma {A}}(x_k-\gamma Bx_k),\\
x_{k+1}= y_k+\gamma (Bx_k-By_k), \;\;k\ge 0,\;\;x_0\in {H}.
\end{cases}
\eeq
\begin{theorem}\label{weakc}
Let Assumptions 1, 2 hold and suppose that the solution set $S:=(A+B)^{-1}(0)\neq \emptyset$. Then the sequence $(x_k)$ generated by Tseng algorithm  with $\gamma^2 L^2\le 1-\varepsilon$ for some $\varepsilon>0$ converges weakly to a solution $\tilde{x}$ of (\ref{main}).
\end{theorem}
\begin{proof}
From (\ref{det}), one has 
\beq\label{discr}
\frac{y_k-x_k}{\gamma}+Bx_k\in -Ay_k,
\eeq
which implies that 
\beq\label{rewr}
\frac{y_k-x_k}{\gamma}+Bx_k-By_k\in -(A+B)y_k.
\eeq
Let $x^*\in S$, i.e., $0\in (A+B)x^*$. 
Using the  monotonicity of $A+B$, we obtain 
\begin{equation}\label{eq0}
\left\langle
\frac{y_k-x_k}{\gamma}+{B}x_k-{B}y_k,
\, y-x^*
\right\rangle \le 0.
\end{equation}
Since $\gamma>0$, from (\ref{det}) and \eqref{eq0} we imply  that
\[
\langle x_{k+1}-x_k,  y_k-x^*\rangle \le 0,
\]
or equivalently 
$$
\langle x_{k+1}-x^*, y_k-x^*\rangle\le \langle x_k-x^*, y_k-x^*\rangle.
$$
Note that 
$$
\langle x_{k+1}-x^*, y_k-x^*\rangle=\frac{1}{2}( \Vert x_{k+1}-x^*\Vert^2+\Vert y_k-x^*\Vert^2-\Vert x_{k+1}-y_k\Vert^2)
$$
and
$$
\langle x_{k}-x^*, y_k-x^*\rangle=\frac{1}{2}( \Vert x_{k}-x^*\Vert^2+\Vert y_k-x^*\Vert^2-\Vert x_{k}-y_k\Vert^2).
$$
Therefore one has 
\baqn
\Vert x_{k+1}-x^*\Vert^2&\le& \Vert x_{k}-x^*\Vert^2+\Vert x_{k+1}-y_k\Vert^2-\Vert x_{k}-y_k\Vert^2\\
&=& \Vert x_{k}-x^*\Vert^2+\gamma^2\Vert Bx_k-By_k \Vert^2-\Vert x_{k}-y_k\Vert^2\\
&\le& \Vert x_{k}-x^*\Vert^2 +\gamma^2 L^2\Vert x_{k}-y_k\Vert^2-\Vert x_{k}-y_k\Vert^2\\
&\le& \Vert x_{k}-x^*\Vert^2-\varepsilon\Vert x_{k}-y_k\Vert^2,
\eaqn
since $B$ is $L$-Lipschitz continuous and  
$$
\gamma^2 L^2\le 1-\varepsilon.
$$
It means that the sequence $(\Vert x_{k}-x^*\Vert)$ is decreasing and hence convergent. Therefore $\Vert x_{k}-y_k\Vert\to 0.$
Let $\bar{x}$ be a weak limit point of $(x_k)$, i.e., there exists a subsequence $(x_{n_k})$ converging to $\bar{x}$. Since $\Vert x_{k}-y_k\Vert\to 0$, we deduce that $(y_{n_k})$ converges weakly to $\bar{x}$ and 
$$
\frac{y_k-x_k}{\gamma}+Bx_k-By_k\to 0.
$$
From (\ref{rewr}) and the fact that $A+B$ is maximally monotone, we imply that $\bar{x}\in S$. Using Opial's Lemma, the conclusion follows. 
\end{proof}
\begin{remark}
i) If $J_{\gamma {A}}$ is well-defined set-valued on $H$, the same conclusion holds. \\

ii) If $L$ is unknown, we can replace the constant step-size $\gamma$ by $\gamma_k>0$ small enough such that 
$$
\gamma_k^2\Vert Bx_k-By_k \Vert^2\le (1-\varepsilon)\Vert x_{k}-y_k\Vert^2,
$$
for each $k\ge 0$ and the conclusion still remains. Note that since $B$ is Lipschitz continuous, there exists $\gamma>0$ such that $\gamma_k>\gamma$ for all $k>0.$
\end{remark}
\subsection{Linear Convergence} 
Next we prove the linear convergence of Tseng algorithm when the maximal monotonicity is enhanced by the strong maximal monotonicity.\\

\noindent$\mathbf{Assumption \;3}:$ The set-valued mapping $A: {H}\rightrightarrows {H}$ is such that $J_{\gamma {A}}$ is well-defined single-valued on $H$ and $A+B$ is maximally $\alpha$-strongly monotone for some $\alpha>0$.\\

\begin{theorem}
Let Assumptions 1, 3 hold. Then the sequence $(x_k)$ generated by   Tseng algorithm with  $\gamma^2 L^2\le 1-2\gamma \alpha$ converges linearly to the unique solution $\tilde{x}$ of (\ref{main}).
\end{theorem}
\begin{proof}
Since $A+B$ is maximally strongly monotone, (\ref{main}) has a unique solution $\tilde{x}$ which satisfies $0\in (A+B)\tilde{x}$. From (\ref{rewr}) and the strong monotonicity of $A+B$, one has 
\begin{equation}
\left\langle
\frac{y_k-x_k}{\gamma}+{B}x_k-{B}y_k,
\, y-\tilde{x}
\right\rangle \le -\alpha \Vert y_k-\tilde{x}\Vert^2
\end{equation}
which deduces that 
\[
\langle x_{k+1}-x_k,  y_k-\tilde{x}\rangle \le - \gamma\alpha  \Vert y_k-\tilde{x}\Vert^2.
\]
Thus 
$$
\langle x_{k+1}-\tilde{x}, y_k-\tilde{x}\rangle\le \langle x_k-\tilde{x}, y_k-\tilde{x}\rangle- \gamma\alpha  \Vert y_k-\tilde{x}\Vert^2.
$$
Similarly as in the proof of Theorem \ref{weakc}, one has 
\baqn
\Vert x_{k+1}-\tilde{x}\Vert^2\le \Vert x_{k}-\tilde{x}\Vert^2 +\gamma^2 L^2\Vert x_{k}-y_k\Vert^2-\Vert x_{k}-y_k\Vert^2- 2\gamma\alpha  \Vert y_k-\tilde{x}\Vert^2.
\eaqn
Since $B$ is $L$-Lipschitz continuous and  
$$
\gamma^2 L^2\le 1-2\gamma \alpha,
$$
we deduce that 
\beq
\Vert x_{k+1}-\tilde{x}\Vert^2\le \Vert x_{k}-\tilde{x}\Vert^2 -2\gamma\alpha(\Vert x_{k}-y_k\Vert^2+ \Vert y_k-\tilde{x}\Vert^2)\le (1-\gamma\alpha)\Vert x_{k}-\tilde{x}\Vert^2.
\eeq
The conclusion follows. 
\end{proof}

\section{Numerical Examples}\label{sec5}
In this section, we illustrate that our result can be used to decomposed efficiently and reduce the Lipschitz constant significantly. In details, we consider the problem $0\in Ax+Bx$ where $A+B$ is maximally monotone and $B$ is non-monotone. The main part $A$ has big Lipschitz constant where the perturbed part $B$ has small Lipschitz constant. In the following examples, we stop if  $\Vert x_{k+1}-x_k\Vert\le 10^{-8}$ or the number of iterations reaches $500$.
 \begin{example}
        First we consider the problem with
        \[
        A=
        \begin{bmatrix}
        8&2&1\\
        2&6&1\\
        1&1&5
        \end{bmatrix},
        \qquad
        Bx=
        \begin{bmatrix}
        0.2\sin(x_1)\\
        0.1\cos(x_3)\\
        0.2\cos(x_2)
        \end{bmatrix}.
        \]
 Then $A+B$ is $4.33$-strongly monotone, $9.67$-Lipschitz continuous  and $B$ is $0.2$-Lipschitz continuous
    We apply  Tseng algorithm to the pair (A,B) with
        $\gamma = 2$ and the initial point
        $x^0=(1,1,1)^T$.
       Figure  \ref{fig:ex5} shows the convergence behavior of the algorithm, which stops after $22$ iterations and can find an approximate solution \[x_1 \approx
        \begin{bmatrix}
        0.00787420\\
        -0.01277598\\
        -0.03901638
        \end{bmatrix}.
        \]

        \begin{figure}[H]
            \centering
            \includegraphics[width=0.67\textwidth]{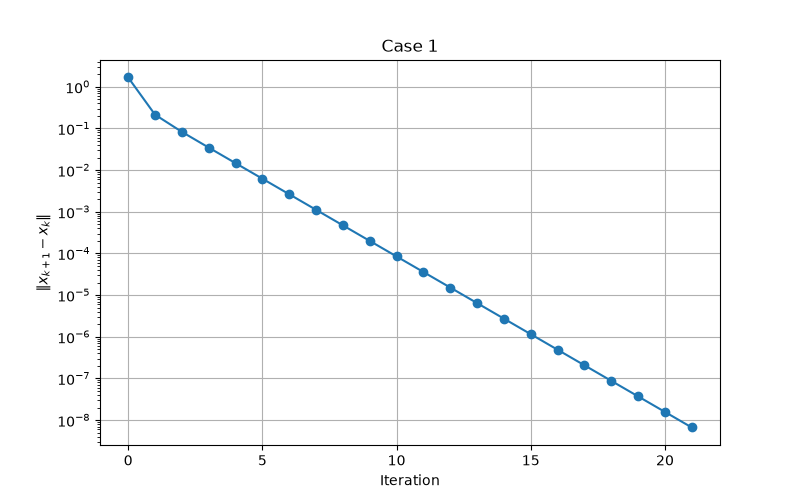}
            \caption{Tseng algorithm applied to the pair $(A,B)$ with
        $\gamma = 2$ }
            \label{fig:ex5}
        \end{figure}
        \noindent Note that since $B$ is non-monotone, classically we have to apply with other decomposition, for example $(A_1,B_1)$ where
        \[
        A_1 = 0,\qquad A_2 = A+B.
        \]
     In this case, we choose the step size $\gamma=0.05$ to satisfy the convergence condition. The algorithm converges slower, stops  after $85$ iterations and can find an approximate solution with less accuracy 
        \[
        x_2 \approx
        \begin{bmatrix}
        0.00787419\\
        -0.01277600\\
        -0.03901634
        \end{bmatrix}.
        \] 

        \begin{figure}[H]
            \centering
            \includegraphics[width=0.67\textwidth]{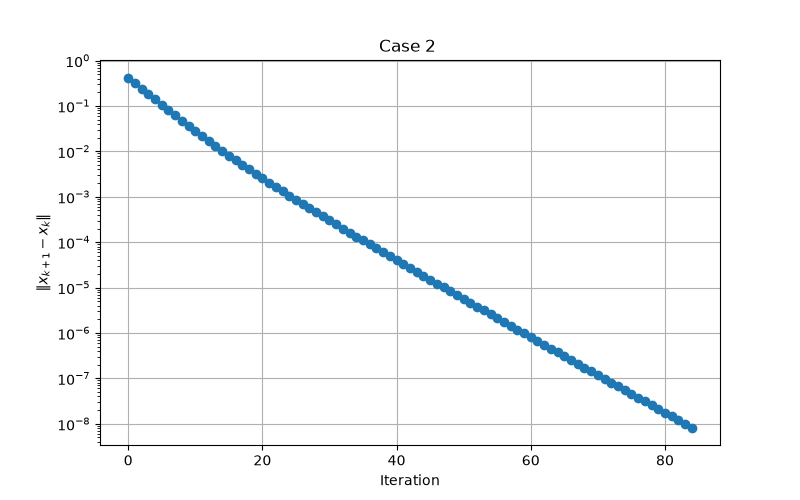}
            \caption{Tseng algorithm applied to the pair $(A_1,B_1)$ with
        $\gamma=0.05$ }    \label{fig:ex6}
        \end{figure}
        \end{example}

   \begin{example} Next, we consider the case where $A$ has big Lipschitz constant, for example with  
    	$$
    	A =\begin{bmatrix} 0.2 & 0 & 0\\ 0 & 5 & 0 \\ 0 & 0 & 100\end{bmatrix},  \quad
    	Bx = \begin{bmatrix} 0.2\sin(x_1) \\ 0.1\cos(x_3) \\ 0.2\cos(x_2)\end{bmatrix}.
    	$$
    
   Similarly, $B$ is non-monotone and $A+B$ is only monotone. We apply Tseng algorithm with 2 pairs $(A,B)$ and $(A_1,B_1)$ where 
    $$
    A_1=0, B_1=A+B.
    $$
    \begin{figure}[H]
        \centering
        \includegraphics[width=0.67\textwidth]{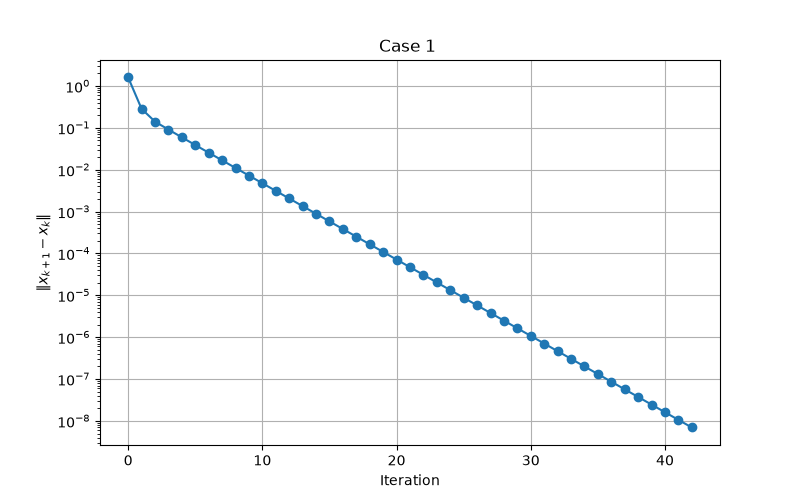}
        \caption{Tseng algorithm applied to the pair $(A,B)$ with
        $\gamma = 2$ }
        \label{fig:ex1}
    \end{figure}
    For this example, we can see the difference clearly. For the first case, we can choose $\gamma$ not  small, for example $\gamma=2$
and  the convergence to an approximate  solution is  rapid (Figure \ref{fig:ex1}).  On the other hand,  if we apply the Tseng algorithm for the pair $(A_1,B_1)$ the Lipschitz constant is  big and thus the step-size $\gamma$ is  small. Consequently,  the iterations reaches the maximum number $500$ but  without finding an approximate  solution. 
        \begin{figure}[H]
        \centering
        \includegraphics[width=0.67\textwidth]{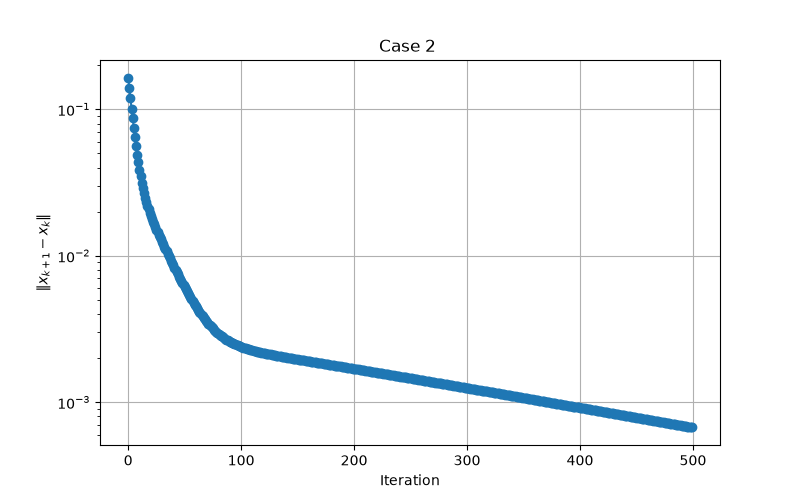}
        \caption{Tseng algorithm applied to the pair $(A_1,B_1)$ with $\gamma=0.008$ }
        \label{fig:ex2}
    \end{figure}
        
        \renewcommand{\arraystretch}{1.5} 
        \begin{table}[H]
          \centering
          \caption{Comparison of   Tseng algorithm applied to $2$ cases }
          \label{tab:ex1}
          \begin{tabular}{lcc}
            \toprule
            \textbf{Case} & {$(A,B)$} & \textbf{$(A_1,B_1)$} \\
            \midrule
            Step Size ($\gamma$)      & $2$         & $0.008$         \\
            \hline
            Iterations & $43$  &$500$\\
            \hline
            Final State ($x_{\text{final}}$) &    $\begin{bmatrix}
            1.33414412\times 10^{-8}\\
            -1.99999600\times 10^{-2}\\
            -1.99960001\times 10^{-3}
            \end{bmatrix}$ & 
            $\begin{bmatrix}
            0.21074297\\
            -0.01999996\\
            -0.00199960
            \end{bmatrix}$\\
            \hline
            Approximate Solution?  & Yes            & No             \\
            \bottomrule
          \end{tabular}
        \end{table}
    \end{example}
    
    \begin{example}
         Now we consider 
        \[
    A=
    \begin{bmatrix}
    0.2 & 0 & 0\\
    0 & 5 & 0\\
    0 & 0 & 10^{10}
    \end{bmatrix},
    \qquad
    Bx=
    \begin{bmatrix}
    0.2\sin(x_1)\\
    0.1\cos(x_3)\\
    0.2\cos(x_2)
    \end{bmatrix}.
    \]
    where the Lipschitz constant of $A+B$ is very big. The first case still converges to an approximate solution very well where the second case run very slow (since the step-size is very small) and the final state is very far from the solutions of the problem.

    \begin{figure}[H]
        \centering
        \includegraphics[width=0.67\textwidth]{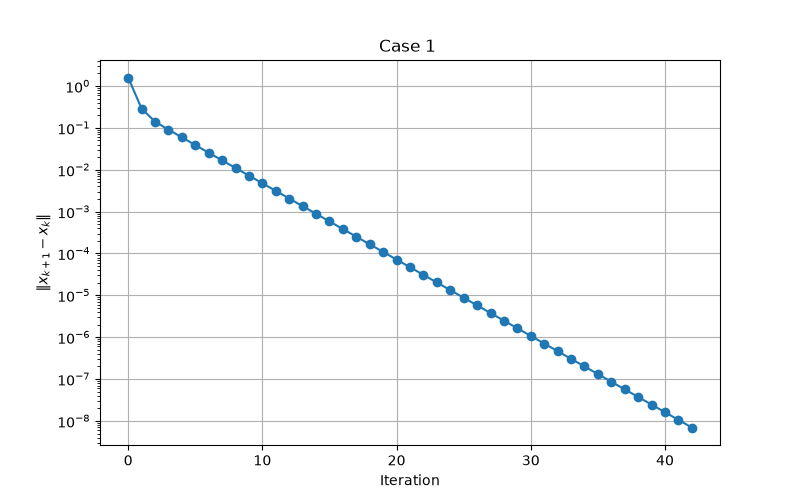}
        \caption{Tseng algorithm applied to the pair $(A,B)$ with $\gamma=2$ }    \label{fig:ex3}
    \end{figure}
    \begin{figure}[H]
        \centering
        \includegraphics[width=0.67\textwidth]{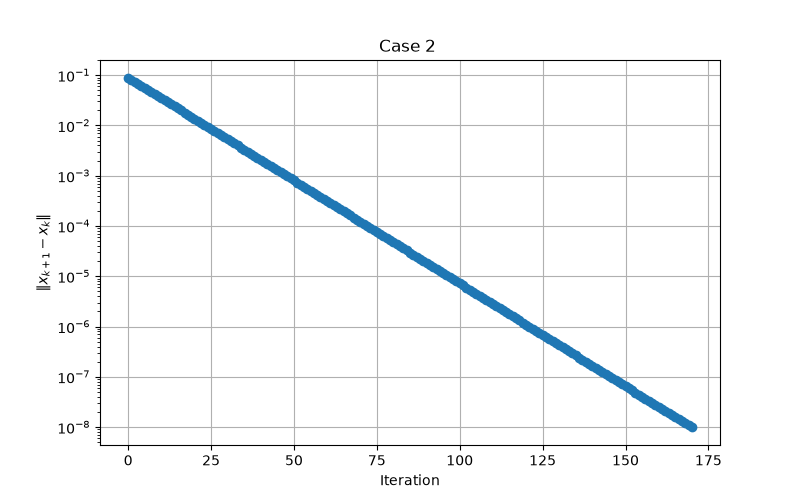}
        \caption{Tseng algorithm applied to the pair $(A_1,B_1)$ with $\gamma=10^{-11}$.}    \label{fig:ex4}
    \end{figure}
        \renewcommand{\arraystretch}{1.5} 
        \begin{table}[H]
          \centering
          \caption{Comparison of of   Tseng algorithm applied to $2$ cases }
          \label{tab:ex2}
          \begin{tabular}{lcc}
            \toprule
            \textbf{Case} & $(A,B)$ & $(A_1,B_1)$ \\
            \midrule
            Step Size ($\gamma$)      & $2$         & $10^{-11}$         \\
            \hline
            Iterations & $43$  &$171$\\
            \hline
            Final State ($x_{\text{final}}$) &    $\begin{bmatrix}
    1.33414412\times10^{-8}\\
    -2.00000000\times10^{-2}\\
    -1.99960001\times10^{-11}
    \end{bmatrix}$ & 
            $\begin{bmatrix}
            9.99999999\times10^{-1}\\
    9.99999991\times10^{-1}\\
    9.90902054\times10^{-8}
            \end{bmatrix}$\\
            \hline
           Approximate Solution?   & Yes            & No             \\
            \bottomrule
          \end{tabular}
        \end{table}
    \end{example}
           
        \begin{example}
            Finally, we consider the inclusion problem
            \[
            0\in N_C(x)+Ax+Bx
            \]
            where
            
            \[ C=[-1,1]^3, \qquad \:\:
            A=
            \begin{bmatrix}
            0.2 & 0 & 0\\
            0 & 5 & 0\\
            0 & 0 & 100
            \end{bmatrix},
            \qquad
            Bx=
            \begin{bmatrix}
            0.2\sin(x_1)\\
            0.1\cos(x_3)\\
            0.2\cos(x_2)
            \end{bmatrix}.
            \]
           We apply Tseng algorithm with 2 pairs $(A_1,B_1)$ and $(A_2,B_2)$ where 
    $$
    A_1=N_C+A, B_1=B
    $$
    and
    
          $$
    A_2=N_C, B_2=A+B
    $$  
    We can use Lemma \ref{sumn} to find the resolvent of $A_1$. For this set-valued example, the first decomposition $(A_1,B_1)$ using our result converges very well where the second decomposition $(A_1,B_1)$ using classical results still does not find an approximate solution.  
            \begin{figure}[H]
                \centering
                \includegraphics[width=0.67\textwidth]{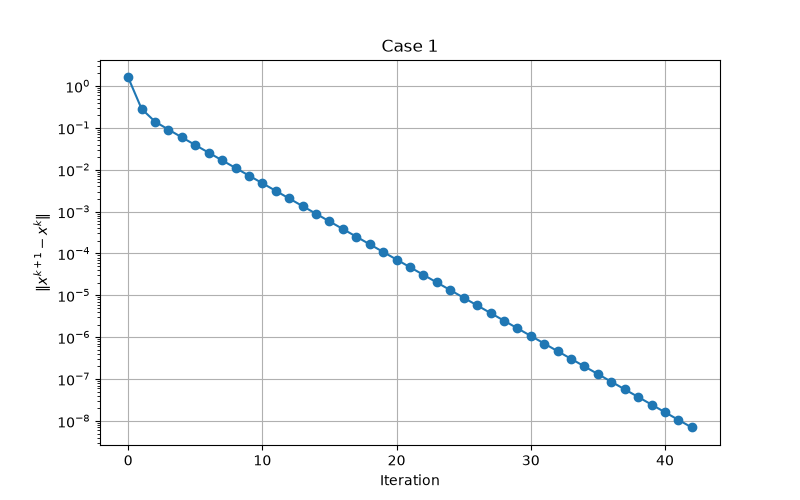}
                \caption{Tseng algorithm applied to the pair $(A_1,B_1)$ with $\gamma=2$ }
                \label{Figure_7}
            \end{figure}

                        \begin{figure}[H]
                \centering
                \includegraphics[width=0.67\textwidth]{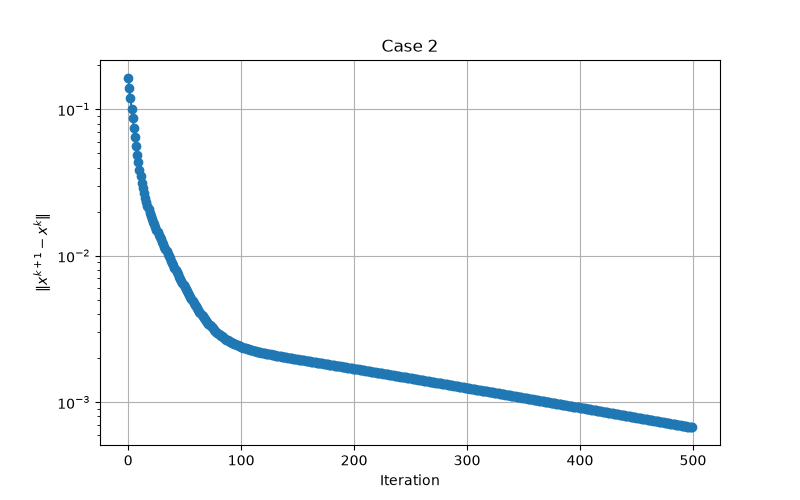}
                \caption{Tseng algorithm applied to the pair $(A_2,B_2)$ with $\gamma=0.008$}
                \label{Figure_8}
            \end{figure}

            \renewcommand{\arraystretch}{1.5}
        \begin{table}[H]
          \centering
          \caption{Comparison of of Tseng algorithm applied to 2 cases}
          \label{tab:ex4}
          \begin{tabular}{lcc}
            \toprule
            \textbf{Case} & $(A_1, B_1)$ & $(A_2, B_2)$ \\
            \midrule
            Step Size ($\gamma$)      & $2$         & $0.008$         \\
            \hline
            Iterations & $43$  &$500$\\
              \hline
            Final State ($x_{\text{final}}$) &    \begin{tabular}[c]{@{}c@{}}($1.33414412\times10^{-8}$, \\ $-0.01999996,$ \\ $-0.00199960$)\end{tabular} & 
    \begin{tabular}[c]{@{}c@{}}($0.21074297$, \\ $-0.01999996$, \\ $-0.00199960$)\end{tabular} \\
            \hline
             Approximate Solution?  & Yes            & No             \\
            \bottomrule
          \end{tabular}
        \end{table}
    \end{example}
\section{Conclusions}\label{sec6}
The paper provides a weak and linear convergence result for the Tseng algorithm applied to the sum of two non-necessarily monotone operators, only provided the monotonicity of the whole sum. Consequently, we can decompose the sum more efficiently and reduce the Lipschitz constant significantly. It would be interesting to extend our result to the sum of three operators, which is a subject of our future research.

\end{document}